\documentclass[11pt]{article}
\usepackage[margin=1in]{geometry} 
\usepackage{amssymb,amsfonts,amsmath,bbm,mathrsfs,stmaryrd}
\usepackage{xcolor}
\usepackage{url}

\usepackage[colorlinks,
             linkcolor=black!75!red,
             citecolor=blue,
             pdftitle={Property (T), factorization and residual finiteness},
             pdfauthor={Angshuman Bhattacharya, Michael Brannan, Alexandru Chirvasitu, Shuzhou Wang},
             pdfproducer={pdfLaTeX},
             pdfpagemode=None,
             bookmarksopen=true
             bookmarksnumbered=true]{hyperref}

\usepackage{braket}
             
\usepackage{tikz}
\usetikzlibrary{arrows,calc,decorations.pathreplacing,decorations.markings,intersections,shapes.geometric,through,fit,shapes.symbols,positioning,decorations.pathmorphing}

\usepackage[amsmath,thmmarks,hyperref]{ntheorem}
\usepackage{cleveref}

\crefname{section}{Section}{Sections}
\crefformat{section}{#2Section~#1#3} 
\Crefformat{section}{#2Section~#1#3} 

\crefname{subsection}{\S}{\S\S}
\crefformat{subsection}{#2\S#1#3} 
\Crefformat{subsection}{#2\S#1#3} 

\theoremstyle{plain}

\newtheorem{lemma}{Lemma}[section]
\newtheorem{proposition}[lemma]{Proposition}
\newtheorem{corollary}[lemma]{Corollary}
\newtheorem{theorem}[lemma]{Theorem}

\theoremstyle{nonumberplain}
\newtheorem{theoremN}{Theorem}

\theoremstyle{plain}
\theorembodyfont{\upshape}
\theoremsymbol{\ensuremath{\blacklozenge}}

\newtheorem{definition}[lemma]{Definition}
\newtheorem{example}[lemma]{Example}
\newtheorem{remark}[lemma]{Remark}

\crefname{definition}{definition}{definitions}
\crefformat{definition}{#2definition~#1#3} 
\Crefformat{definition}{#2Definition~#1#3} 

\crefname{ex}{example}{examples}
\crefformat{example}{#2example~#1#3} 
\Crefformat{example}{#2Example~#1#3} 

\crefname{remark}{remark}{remarks}
\crefformat{remark}{#2remark~#1#3} 
\Crefformat{remark}{#2Remark~#1#3} 

\crefname{convention}{convention}{conventions}
\crefformat{convention}{#2convention~#1#3} 
\Crefformat{convention}{#2Convention~#1#3}

\crefname{lemma}{lemma}{lemmas}
\crefformat{lemma}{#2lemma~#1#3} 
\Crefformat{lemma}{#2Lemma~#1#3} 

\crefname{proposition}{proposition}{propositions}
\crefformat{proposition}{#2proposition~#1#3} 
\Crefformat{proposition}{#2Proposition~#1#3} 

\crefname{corollary}{corollary}{corollaries}
\crefformat{corollary}{#2corollary~#1#3} 
\Crefformat{corollary}{#2Corollary~#1#3} 

\crefname{theorem}{theorem}{theorems}
\crefformat{theorem}{#2theorem~#1#3} 
\Crefformat{theorem}{#2Theorem~#1#3} 

\crefname{assumption}{assumption}{Assumptions}
\crefformat{assumption}{#2assumption~#1#3} 
\Crefformat{assumption}{#2Assumption~#1#3} 

\crefname{equation}{}{}
\crefformat{equation}{(#2#1#3)} 
\Crefformat{equation}{(#2#1#3)}

\theoremstyle{nonumberplain}
\theoremsymbol{\ensuremath{\blacksquare}}

\newtheorem{proof}{Proof}

\newtheorem{proof_of_colim}{Proof of \Cref{pr.colim}}
\newtheorem{proof_of_RFDFP}{Proof of \Cref{th.RFD_FP}}
\newtheorem{proof_of_tf}{Proof of \Cref{th.tf}}
\newtheorem{proof_of_brf}{Proof of \Cref{th.bicr-rf}}

\newcommand\bC{\mathbb C}

\newcommand\bG{\mathbb G}

\newcommand\bZ{\mathbb Z}

\newcommand\cA{\mathcal A}
\newcommand\cB{\mathcal B}
\newcommand\cC{\mathcal C}

\newcommand\cP{\mathcal P}

\DeclareMathOperator{\id}{id}
\DeclareMathOperator{\ir}{\mathrm{Irred}}

\DeclareMathOperator{\tr}{\mathrm{tr}}
\DeclareMathOperator{\Tr}{\mathrm{Tr}}

\newcommand{\qedhere}{\mbox{}\hfill\ensuremath{\blacksquare}}

\title{\bf Property (T), property (F) and residual finiteness for \\ 
	discrete quantum groups} 
\author{Angshuman Bhattacharya\footnote{Department of Mathematics, Indian Institute of Science Education and Research Bhopal, Indore By-pass Road, Bhauri Bhopal, 462066, Madhya Pradesh, India \url{angshu@iiserb.ac.in}}, 
	Michael Brannan\footnote{Department of Mathematics, Texas A\&M University, College Station, TX 77843, USA \url{mbrannan@math.tamu.edu}}, \\ 
	Alexandru Chirvasitu\footnote{Department of Mathematics, University at Buffalo, Buffalo, NY 14260, USA \url{achirvas@buffalo.edu}}, 
	Shuzhou Wang\footnote{Department of Mathematics, University of Georgia, Athens, GA 30602, USA \url{szwang@uga.edu}}}

\begin{document}

\date{}

\maketitle

\begin{abstract}
  We investigate connections between various rigidity and softness properties for discrete quantum groups. After introducing a notion of residual finiteness, we show that it implies the Kirchberg factorization property for the discrete quantum group in question. We also prove the analogue of Kirchberg's theorem, to the effect that conversely, the factorization property and property (T) jointly imply residual finiteness. We also apply these results to certain classes of discrete quantum groups obtained by means of bicrossed product constructions and study the preservation of the properties (factorization, residual finiteness, property (T)) under extensions of discrete quantum groups.
\end{abstract}


\vspace{.5cm}


\tableofcontents

\section*{Introduction}

The theory of compact quantum groups initiated by Woronowicz in \cite{Wor87} has proven very flexible and amenable to treatment from multiple perspectives. The objects in loc. cit. are (particularly well-behaved) Hopf algebras, usually regarded as function algebras on the non-commutative geometer's version of a compact group. On the other hand, by Pontryagin duality for quantum groups, the same algebras are group algebras of their {\it discrete} quantum group duals \cite{PoWo90,BS93,EfRu,dae-discr}. We adopt the latter perspective herein, studying the interaction between the quantum versions of several properties of interest in the representation theory of discrete groups and in geometric group theory.

The motivating classical result for the present note is the main result of Kirchberg in \cite[Theorem 1.1]{kirch}, stating that a discrete group with property (T) and property (F) (also known as the {\it factorization property}) is residually finite. Our main result here is a quantum version thereof (see \Cref{th.tf}).

\begin{theoremN}
{\bf 1}
A discrete quantum group with property (T) and property (F) is residually finite.
\end{theoremN}

This is accompanied by another analogue of a classical result (\Cref{th.RFD_FP}):

\begin{theoremN}
{\bf 2}
  Residually finite discrete quantum groups have property (F).
\end{theoremN}

Theorem 2, together with the main theorem in Chirvasitu \cite{chi-rf},  
implies the main results in both Bhattacharya-Wang \cite{BW}
and Brannan-Collins-Vergnioux \cite{bcv}, which state that, respectively, the discrete duals of 
the universal unitary quantum groups $U_n^+$ and orthogonal quantum groups $O_n^+$ 
have Kirchberg's property (F) and (therefore) the Connes embedding property when $n \neq 3$,  
though it is believed that the same assertions hold for $n=3$ as well. 
  
We will now unpack the ingredients going into Kirchberg's original
result and its quantum version in Theorem 1 and the related Theorem 2 above, with a more detailed exposition below. First, the property (T) for locally compact groups was originally introduced in \cite{kazhdan} 
and it has had far reaching impact in group theory, ergodic theory and operator algebras. 
Some of these achievements can be found in excellent references \cite{hv-t,t,zim}.  
Property (T) is a representation-theoretic {\it rigidity}
property, to the effect that, for a discrete group, the trivial representation 
 is isolated in the set of all irreducible representations 
with respect to the topology of pointwise convergence for the associated positive definite 
functions; it has several equivalent formulations 
(see the above-cited references and the recollection in \Cref{se.prel} below).

Property (T) was adapted to the setting of discrete {\it quantum}
groups in \cite{fim} and its meaning in the statement of Theorem 1 above is taken in 
this new context. We note that a discrete quantum group with property (T) necessitates that its  
antipode be bounded, a condition which is equivalent to the discrete quantum group being unimodular. 
Property (T) has been discussed in this context and the more general locally compact quantum 
setting in a number of other works (e.g. \cite{ky,ky-so, DFSW16,chen-ng,dsv, BrKe17}).

The second half of the hypothesis, property (F), of Theorem 1 is sometimes also
referred to as the {\it factorization property} (or indeed the
Kirchberg factorization property). It was introduced for locally compact groups 
by Kirchberg in \cite{kirch-orig} and further studied in \cite{kirch-ss,kirch}. This 
property amounts, for a discrete group $\Gamma$, to requiring that the
representation
\begin{equation*}
  \begin{tikzpicture}[auto,baseline=(current  bounding  box.center)]
    \path[anchor=base] (0,0) node (cc) {$C^*(\Gamma)\otimes_{\mathrm{max}}C^*(\Gamma)^{op}$} +(4,0) node (b) {$B(\ell^2(\Gamma))$};
    \draw[->] (cc) -- (b);
  \end{tikzpicture}
\end{equation*}
resulting from the left and right translation actions of $\Gamma$ on
itself factors as 
\begin{equation*}
  \begin{tikzpicture}[auto,baseline=(current  bounding  box.center)]
    \path[anchor=base] (0,0) node (cc) {$C^*(\Gamma)\otimes_{\mathrm{max}}C^*(\Gamma)^{op}$} +(8,0) node (b) {$B(\ell^2(\Gamma))$} +(4,-.5) node (min) {$C^*(\Gamma)\otimes_{\mathrm{min}}C^*(\Gamma)^{op}$};
    \draw[->] (cc) to[bend left=6] (b);
    \draw[->] (cc) to[bend right=6] (min);
    \draw[->] (min) to[bend right=6] (b);
  \end{tikzpicture}
\end{equation*}
(where $C^*(\Gamma)$ denotes the {\it full} group $C^*$-algebra). The
property has several alternative characterizations; among them is the
requirement that the group admits a ``sufficiently large'' family of
unital completely positive (UCP) maps  $C^*(\Gamma)\to M_n(\bC)$ for 
increasing $n$ that are ``almost representations'' (see
e.g. \cite[Proposition 3.2]{kirch}, \cite[Theorem 6.2.7]{bo},
\cite[Theorem 6.1]{oz} and the discussion below, in \Cref{se.ft}).

This last reformulation of property (F) makes it possible to interpret the
latter as a ``softness'' property, ensuring that the discrete group is
approximable by small (linear, roughly speaking) quotients. Property
(F) is considered in the wider context of discrete quantum groups in
\cite{BW}, which provided another motivation for the present paper.

Finally, the conclusion of the above-cited theorem and the condition in Theorem 2 refer to residual finiteness. For a discrete group $\Gamma$ this simply means that every non-trivial element $\gamma\in \Gamma$ has non-trivial image in some finite quotient of $\Gamma$ (i.e. $\Gamma$ has ``enough'' finite quotients). Residual finiteness is widely studied to the extent that we cannot do the literature justice here.

Several generalizations of the notion of residual finiteness can be defined for discrete quantum groups. One of them is taking the fact that finitely generated linear groups are residually finite \cite{mal-lin} as a cue since our main interest in the quantum setting are noncommutative analogs of them; in the presence of finite generation residual finiteness for the discrete group $\Gamma$ can be recast as the requirement that the group $*$-algebra $\bC\Gamma$ have enough $*$-representations on finite-dimensional Hilbert spaces.  In this paper, we make the analogue of this requirement as the defining property for residual finiteness in the quantum case (see \Cref{se.rf} as well as \cite{chi-rf} for a precursor to this).

The interpretation of property (F) given above, in terms of finite-dimensional almost representations, makes Kirchberg's theorem very intuitive: in the presence of the rigidity property (T) the almost representations become honest representations of $\Gamma$ on finite-dimensional Hilbert spaces. The proof of \cite[Theorem 1.1]{kirch} captures this intuition, as does the proof of the analogous quantum statement in \Cref{th.tf} below.

The paper is organized as follows.  \Cref{se.prel} is devoted to recalling some of the necessary background for the sequel, including more precise formulations for the properties referred to above.  In \Cref{se.rf} we argue that residual finiteness implies property (F) for discrete quantum groups, as expected. \Cref{th.tf} of \Cref{se.ft} is the main result of this note, proving that the analogue of Kirchberg's theorem holds in the quantum setting. Apart from whatever intrinsic interest this might hold, it is perhaps a good indication that the notion of residual finiteness adopted here for (finitely generated) discrete quantum groups is the ``right one'', and well suited for further exploration. Finally, in \Cref{se.ap} we investigate the behavior of the various properties studied here under extensions of discrete quantum groups (see \Cref{def.ext} for the notion of extension) and give some applications of Theorems 1 and 2 for this setting.  As is the case classically, property (T) is preserved under extensions. We also prove a partial positive result in the same spirit for residual finiteness of discrete quantum groups in \Cref{th.bicr-rf} and for property (F) in \Cref{pr.fext}.

\subsection*{Acknowledgements}

M.B. is grateful for funding through NSF grant DMS-1700267. A.C. was partially supported by NSF grants DMS-1565226 and DMS-1801011.

\section{Preliminaries}\label{se.prel}

For the basics on compact and discrete quantum groups and their duality, we refer the reader to the standard references \cite{Ti08, Wor98, PoWo90, EfRu, dae-discr}.

For a discrete quantum group $\Gamma$ with compact quantum dual $G=\widehat{\Gamma}$ we typically denote its group algebra $\bC \Gamma$ by $\cA=\cA(G)$; this is the CQG algebra of representative functions on $G$. $\ir(G)$ denotes the set of irreducible (and hence finite-dimensional) representations of $G$; these are in one-to-one correspondence with the simple comodules of $\cA(G)$.

Throughout the paper we use the terms `representation of $G$', `corepresentation of $\cA(G)$' and `comodule of $\cA(G)$' interchangeably. 

For each $x\in \ir(G)$ representing an $n$-dimensional irreducible unitary representation $V$ of $G$ we have a matrix $u^x=(u^x_{ij})_{1\le i,j\le n}$, unitary in $M_n(\cA)$, consisting of matrix counits spanning the smallest subcoalgebra $C\subset \cA$ for which $V\to V\otimes \cA$ factors through $V\otimes C$. We sometimes denote
\begin{equation*}
  C^x = \mathrm{span}\{u^x_{ij}\},\ u^x = (u^x_{ij})\in M_n(\cA), 
\end{equation*}
and the underlying Hilbert space of the irreducible representation $x$ by $H_x$.

In \Cref{se.ft} below we will make use of property (T) for the discrete quantum group $\Gamma=\widehat{G}$. For background on this (some of which we recall below) we will be referring mainly to \cite{fim} (especially Section 3 therein). As for the classical property (T) for discrete groups, the reader can consult \cite{t} for a rather comprehensive account.

\subsection{Property (T)}

Let us recall (see \cite[Definition 3.1]{fim}) the following definition.

\begin{definition}\label{def.t}
  Let $G$ be a compact quantum group with underlying CQG algebra $\cA$, $X\subset \ir(G)$ a subset, and $\pi:\cA\to B(H)$ a $*$-representation on a Hilbert space $H$. For $x \in \ir(G)$, put $U^x = (\id \otimes \pi )u^x \in B(H_x) \otimes B(H)$.

\begin{enumerate}
\item
For $\varepsilon>0$ we say that the unit vector $v\in H$ is {\it
  $(X,\varepsilon)$-invariant} if for all $x\in X$ and all non-zero
$\eta\in H_x$ we have
\begin{equation*}
  \|U^x(\eta\otimes v)-(\eta\otimes v)\| < \varepsilon \|\eta\|.
\end{equation*}
\item We say the representation $\pi$ {\it almost contains invariant vectors}
if there are $(X,\varepsilon)$-invariant vectors for all finite
subsets $X\subseteq \ir(G)$ and all $\varepsilon>0$. In that case we
write ${\bf 1}\preceq \pi$.
\item We say $\Gamma=\widehat{G}$ {\it has property (T)} if whenever
${\bf 1}\preceq \pi$ the representation $\pi$ contains the
trivial representation as a summand (i.e. ${\bf 1}\le \pi$): there
exists a unit vector $v\in H$ such that
\begin{equation*}
  U^x(\eta\otimes v)=\eta\otimes v
\end{equation*}
for all $x\in \ir(G)$ and all $\eta\in H_x$.
\end{enumerate}
\end{definition}

\begin{remark}\label{re.1phi}
  The property ${\bf 1}\preceq \pi$ is sometimes also expressed by
  saying that $\pi$ {\it weakly contains} the trivial
  representation ${\bf 1}$ of $\Gamma$.   It is easy to see that the condition ${\bf 1}\preceq \pi$
  defined as above is equivalent to the existence of a net $(v_n)_n$
  of unit vectors in $H$ such that for every $x\in \ir(G)$ and every
  unit vector $\eta\in H_x$ we have
  \begin{equation*}
    U^x(\eta\otimes v_n)-(\eta\otimes v_n)\to 0 \qquad \text{in norm}. 
  \end{equation*}
\end{remark}

\begin{remark}\label{re.cont-const}
  We will need an ostensibly stronger but in fact equivalent formulation of property (T).

  First, \cite[Proposition 3.4]{fim} shows that a single $X$ and $\varepsilon$ suffice: if the group has property (T) as in \Cref{def.t} then there exist a finite set $X$ and $\varepsilon>0$ such that every representation with $(X,\varepsilon)$-invariant vectors contains invariant vectors.

  Secondly, we can then improve on this further so as to ensure invariant vectors are close to almost invariant ones: given a $\delta>0$ we can choose $X$ and $\varepsilon$ such that every $(X,\varepsilon)$-invariant unit vector is $\delta$-close to a unit invariant vector.

  This latter version is sometimes referred to as `property (T) with {\it continuity constants}' and appears explicitly in \cite[Proposition 1.16]{hv-t}. That proof (or that of \cite[Lemma 2.1]{kirch}, applied to the unitaries $U^x$) can be adapted to the quantum setting to yield the equivalence of the two definitions of property (T) (with and without continuity constants). See also \cite{pp} for an extended discussion of the matter in the context of von Neumann algebras.  
\end{remark}

Given that for a compact quantum group $G$ we regard $\cA(G)$ as the group algebra $\bC \Gamma$ of the discrete quantum dual $\Gamma=\widehat{G}$, the following concept is natural.

\begin{definition}\label{def.fg}
  Let $\Gamma=\widehat{G}$ be a discrete quantum group. A subset $X \subset \Gamma$
  {\it generates} $\Gamma$ if the matrix counits $u^x_{ij}$ generate
  $\cA(G)$ as a $\ast$-algebra.  We say that  $\Gamma$ is {\it finitely generated} if some finite subset   $X\subseteq \ir(G)$ generates $\Gamma$.
\end{definition}

On occasion, we refer to finitely generated CQG algebras as {\it CMQG
  algebras}.

\begin{remark}
  \Cref{def.fg} is equivalent to the notion of finite generation in
  \cite[$\S$2.3]{fim} and \cite[Theorem 2.5]{Wan95_1}.
\end{remark}

The relevance of \Cref{def.fg} to the present paper is that the finite
generation property is implied by property (T) (see the quantum
analogue \cite[Proposition 3.3]{fim} of the classical result to the same
effect, e.g. \cite[Theorem 1.3.1]{t}):

\begin{proposition}\label{pr.tfg}
  A discrete quantum group with property (T) is finitely
  generated. \qedhere
\end{proposition}

\begin{remark}
Any discrete quantum group $\Gamma$ with property (T) is also known to be of {\it Kac type} (or {\it unimodular}) \cite{fim}.  The latter means that the left and right Haar weights on $\Gamma$ are equal and tracial. It has two other equivalent forms, the compact dual quantum group $G$ has tracial Haar state and the antipode on 
$\cA(G)$ is bounded for the universal $C^*$-norm.
\end{remark}

\subsection{Property (F)}

We now briefly review the Kirchberg factorization property for discrete quantum groups, which was introduced and studied  in \cite{BW}.

Let $A$ be a unital C$^\ast$-algebra and $\tau:A \to \bC$ a tracial state with GNS triple $(\pi_\tau, H_\tau, \Lambda_\tau(1))$, where $\Lambda_\tau:A \to H_\tau$ is the canonical morphism from $A$ to the GNS Hilbert space $H_\tau$. Denote by $\pi_\tau^{op}$ the representation of the opposite C$^\ast$-algebra $A^{op}$ of $A$ on $H_\tau$ defined by 
\[
\pi_\tau^{op}(a^{op})\Lambda_\tau(b) = \Lambda_\tau(ba) \qquad (a,b \in A).
\]
Since $\pi_\tau$ and $\pi_\tau^{op}$ are commuting representations of $A$ and $A^{op}$, respectively, we obtain a representation of the {\it maximal} C$^\ast$-algebra tensor product 
\[(\pi_\tau \cdot \pi_\tau^{op})_{\max}:A \otimes_{\max} A^{op} \to B(H_\tau); \qquad  a \otimes b^{op} \mapsto \pi_\tau(a)\pi_\tau^{op}(b^{op}) \qquad (a,b \in A). \]

In the following, the normalized trace on the $k \times k$ matrix algebra $M_k = M_k(\bC)$ is denoted by $\text{tr}_k$.

\begin{theorem} \label{thm-f}
	{\rm (See \cite[Proposition 3.2]{kirch}, \cite[Theorem 6.1]{oz} and \cite[Theorem 6.2.7]{bo})}\label{amenable-trace-fp}   \label{kfp-equivalences}
For a trace $\tau$ on a C$^\ast$-algebra $A \subseteq B(H)$,  the following are equivalent.
\begin{enumerate} 
\item $\tau$ extends to an $A$-central state $\psi \in B(H)^*$.  I.e., $\psi(uxu^*) = \psi(x)$ for each $x \in B(H)$ and each unitary $u \in A$.
\item There is a net of unital and completely positive (abbreviated UCP) maps $\varphi_k:A\to M_{n_k} $ such that $\tr_{n_k}\circ \varphi_k (a)\to \tau(a)$ and 
$  \|\varphi_k(a^*b)-\varphi_k(a)^*\varphi_k(b)\|_{2,n_k}\to 0
$ for each $a, b \in A$, where $\|x\|_{2,n}=\tr_n(x^*x)^{\frac 12}$ for $x\in M_n$. 
\item The representation $(\pi_\tau \cdot \pi_\tau^{op})_{\max}:A \otimes_{\max} A^{op} \to B(H_\tau)$ factors through the quotient \\
$A \otimes_{\max} A^{op} \to A \otimes_{\min} A^{op}$.
\end{enumerate}
\end{theorem}

Note that property 1 in the above does not 
depend on the choice of embedding  $A \subseteq B(H)$. 

\begin{definition}
Any tracial state $\tau:A \to \bC$ satisfying the hypotheses of the above theorem is called {\it amenable}.
\end{definition}

\begin{remark} Note that for $\tau: \cA \to \bC$ to be amenable, it suffices to check condition 2 on any norm-dense $\ast$-subalgebra $\cA \subseteq A$.  In particular, if $G$ is a compact quantum group and $C^u(G) = C^*(\cA(G))$ denotes the universal enveloping C$^*$-algebra of the CQG-algebra $\cA (G)$, we shall call a tracial state $\tau:\cA(G)\to \bC$ {\it amenable} if its unique extension to $C^u(G)$ is amenable. 
\end{remark}

\begin{definition}\label{def.f} 
  Let $\Gamma=\widehat{G}$ be a discrete quantum group of Kac type. We say that $\Gamma$ has
  {\it the Kirchberg factorization property} (or {\it is FP}, or {\it
    has property (F)}) if the Haar trace on $\cA(G)$ is amenable.
\end{definition}

\begin{remark}This notion is precisely as in \cite[Definition 2.10]{BW}, whose  authors are investigating extensions of this property to non-unimodular discrete quantum groups where the 
Tomita-Takesaki theory is essential.  
\end{remark}

\subsection{Residual finiteness}

The notion of residual finiteness for discrete groups has several possible generalizations to discrete quantum groups. 
In this paper, we will use the following definition. 

\begin{definition}\label{def.rf}
  A discrete quantum group $\Gamma=\widehat{G}$ is called {\it RFD} if its
  underlying CQG algebra $\cA=\cA(G)$ embeds as a $\ast$-algebra into a product of matrix algebras.  I.e., 
   if for any $0\ne a\in \cA$ there is some $*$-representation $\pi$ of $A$ on a finite dimensional Hilbert space such that $\pi(a)\ne 0$ 
   (we then also say that $\cA$ itself is RFD).   If in addition $\Gamma$ is finitely generated in the sense of
  \Cref{def.fg}, then we say that it is {\it residually finite} (or
  {\it RF} for short).
\end{definition}

We note that any RFD discrete quantum group $\Gamma$ is automatically of Kac type.  See for example \cite[Remark A.2]{So05}.

\begin{remark} 
The above definition specializes to the classical notion of residual finiteness when the 
discrete quantum group in question is a finitely generated discrete group (since finitely generated maximally almost periodic groups are well-known to be residually finite). Each of the following five (a priori stronger) conditions on a CQG algebra $\cA$ also restricts to the usual notion of residual finiteness for classical discrete groups. Therefore each 
would deserve a name reflecting ``residual finiteness'' for genuine quantum groups. 

\begin{enumerate}
	
\item 
The first condition is demanding more than in \Cref{def.rf}, reflecting the original notion of residual finiteness in group theory which requires there to be sufficiently many finite group quotients: 
there is a faithful family of Hopf $\ast$-algebra morphisms $\pi_n: \cA \to H_n$ from the CQG algebra $\cA$ onto
(not necessarily co-commutative) finite-dimensional Hopf $\ast$-algebras $H_n$.

\item 
In addition to the condition (1),    
for each finite family of irreducible corepresentations $u^{\alpha_1}, ..., u^{\alpha_k}$ of  $\cA$, there exists a $\pi_{n_0}$ (from among the $\pi_n$'s) 
such that the corepresentation $(\id \otimes \pi_{n_0}) (u^{\alpha_1}),$ $\ldots, (\id \otimes \pi_{n_0}) (u^{\alpha_k})$ of 
the Hopf algebras $H_{n_0}$ are irreducible. (This condition extends \cite[Lemma 3.7.9]{bo} for 
residually finite discrete groups.)

\item 
In addition to the condition (1), require $\pi_n$ there to be co-normal morphisms. 
(Recall that according to \cite[Theorem 2.7]{Wan13}, 
when $\cA = \cA (G)$ and $H_n = \cA (N)$, the surjection 
$\pi_n: \cA \to H_n$ being co-normal \cite{MontS} is equivalent to $N$ being a compact normal quantum subgroup.) 

\item 
In addition to  the conditions in (2), require $\pi_n$ to be co-normal morphisms.

\item 
There is a family of cofinite dimensional normal Hopf $\ast$-subalgebras $H_n$
of $\cA(G)$ whose intersection is the trivial Hopf algebra (i.e. the scalar field).
Here a normal Hopf $\ast$-subalgebra is called cofinite if the quotient 
$\cA(G)/\cA(G){H^+_n}$ is a finite dimensional Hopf algebra, where $H_n^+$ is the kernel of the counit of $H_n$.

\end{enumerate}

More work needs to be done to investigate further examples beyond discrete groups regarding the above properties as well as deeper results beyond these properties. 
For instance, it would be interesting to determine if any of 
the known quantum groups ($q$-deformed ones as well as the universal or free ones) satisfy any of the conditions above. 

\end{remark}

\section{Residual finiteness implies property (F)}\label{se.rf}

The main results of this section is that the RFD property for a not necessarily finitely generated discrete quantum group implies property (F):

\begin{theorem}\label{th.RFD_FP}
  An RFD discrete quantum group has property (F).
\end{theorem}

Before going into the proof, we will make some preparations. First, we
reduce the problem to Pontryagin duals of compact {\it matrix} quantum
groups in the sense of \cite{Wor87} (where they are referred to as
`compact matrix pseudogroups'). Recall that these are compact quantum
groups $G$ whose discrete quantum duals are finitely generated in the
sense of \Cref{def.fg}. That is, the underlying CQG algebra $\cA(G)$
is finitely generated (as an algebra, or equivalently, as a
$*$-algebra).

For every compact quantum group $G$, we can write the corresponding
CQG algebra $\cA(G)$ as the union of its CMQG subalgebras $\cA(G_i)$
(for $i$ ranging over some index set). For this reason, the following
result is relevant to our specialization to compact matrix quantum
groups.

\begin{proposition}\label{pr.colim}
  Let $G$ be a compact quantum group, and suppose $\cA(G)$ can be
  written as the union $\varinjlim_i \cA(G_i)$ of CQG subalgebras for
  a family of quantum group quotients $G\to G_i$.   Then, $\widehat{G}$ has property (F) if and only if each $\widehat{G_i}$ has property (F).
\end{proposition}

This will be an immediate consequence of the following more general result.

\begin{proposition}\label{pr.colim_gen}
  Let $A$ be a $C^*$-algebra expressible as a filtered inductive limit $\varinjlim_i A_i$ of $C^*$-algebras. Let also $\tau$ be a trace on $A$.  Then, $\tau$ is amenable if and only if its restrictions $\tau_i=\tau|_{A_i}$ are all amenable. 
\end{proposition}
\begin{proof}
  One direction is immediate: the amenability for the $\tau_i$ follows from the fact that a net of UCP maps $\varphi_k: A\to M_{n_k}$ witnessing amenability for $\tau$ restrict to UCP maps $A_i\to M_{n_k}$ witnessing the amenability of each $\tau_i$.

Conversely, suppose all $\tau_i$ are amenable. We will prove that $\tau$ is amenable by means of characterization (1) in \Cref{kfp-equivalences}: for an embedding $A\subseteq B(H)$, $\tau$ extends to an $A$-central state on $B(H)$.  Note that since the amenability of a trace, a priori, is defined only in terms of the GNS representation of the trace, the cited characterization goes through so long as $A\to B(H)$ is a representation whose kernel is contained in that of the trace. 

In conclusion, the amenability of the $\tau_i$ implies the existence of $A_i$-central states $\psi_i$ on $B(H)$, where the $A_i$ map into the latter via the compositions
\begin{equation*}
  A_i\to A\to B(H). 
\end{equation*}
Now let $\psi$ be a the state on $B(H)$ obtained as the limit of some w$^\ast$-convergent sub-net of $(\psi_i)_i$. It follows immediately from its construction that $\psi$ is an $A$-central extension of $\tau$ to $B(H)$, hence the conclusion.  
\end{proof}

\begin{remark}
  Note that neither the structure maps $A_i\to A=\varinjlim_i A_i$ of the  inductive limit nor the connecting maps $A_i\to A_j$ are assumed to be one-to-one.  
\end{remark}

\begin{proof_of_colim}
Simply apply \Cref{pr.colim_gen} to the universal $C^*$-algebra $C^u(G) = C^*(\cA(G))$ associated to $G$, expressed as the filtered  inductive limit of the universal $C^*$-algebras $C_i$ associated to the compact matrix quantum quotients $G\to G_i$. The trace $\tau$ in question here is the Haar state of $C^u(G)$, which indeed, as \Cref{pr.colim_gen} requires, restricts to the Haar states of $\tau_i:C^u(G_i) \to \bC$. 
\end{proof_of_colim}

In conclusion, we get

\begin{corollary}\label{cor.reduction}
  If the statement of \Cref{th.RFD_FP} holds for duals of compact matrix
  quantum groups, then it holds in general. 
\end{corollary}
\begin{proof}
Let $G$ be a compact quantum group with the property that $\cA(G)$ is
residually finite-dimensional. As noted above, $\cA(G)$ is the union of $\cA(G_i)$ as $G_i$ range
over the compact matrix quantum group quotients $G\to G_i$. Residual
finite-dimensionality is inherited by $*$-subalgebras, so we know that
all $\widehat{G_i}$ are RFD and hence, by assumption,
have property (F). \Cref{pr.colim} now finishes the proof. 
\end{proof}

We are now ready for the proof of the main result announced above.

\begin{proof_of_RFDFP}
According to \Cref{cor.reduction}, it suffices to assume that the
discrete quantum group in question is $\widehat{G}$, where $G$ is a
compact matrix quantum group. 

In this setup, the countable dimensionality of $\cA=\cA(G)$ as a complex
$*$-algebra, together with the RFD property, ensure that we have an
embedding 
\begin{equation}\label{eq:embedding}
  \cA \hookrightarrow M:=\prod_{k=1}^\infty M_{n_k}
\end{equation}
into a {\it countable} product of matrix algebras. Here, the right
hand side of \Cref{eq:embedding} signifies the product in the category
of $C^*$-algebras, i.e. the set of {\it bounded} sequences of elements
in $M_{n_k}$ as $k$ ranges over the positive integers. 

Now consider a sequence $\alpha_k>0$, $k\ge 1$ of positive reals
adding up to $1$, and let
\begin{equation*}
  \tau=\sum_{k=1}^\infty \alpha_k \tr_{n_k}
\end{equation*}
be the corresponding faithful state on $M$ (where $\tr_n$ denotes the
normalized trace on $M_n$). By \Cref{pr.colim_gen}, 
$\tau$ is an amenable trace on $M$. 

We regard $\cA$ as a $*$-subalgebra of $M$ via \Cref{eq:embedding},
and by a slight abuse of notation we regard $\tau$ as a state on
$\cA$. The proof of \cite[Proposition 4.1]{Wor87} shows that the
Ces\`aro limit of the convolution iterates $\tau^{*n}$ is precisely
the Haar state on $\cA$ (note that although Woronowicz requires
faithfulness of $\tau$ on a $C^*$ completion of $\cA$, the proof only
requires this on $\cA$).

The conclusion now follows from the observations that (a) $\tau$ is an
amenable trace on $\cA$ by \cite[Proposition 6.3.5.(a)]{bo} and (b) the collection of amenable traces is
weak$^*$-closed, and also closed under convolution and convex
combinations (\cite[Propositions 2.12, 2.13]{BW}). 
\end{proof_of_RFDFP}

\begin{remark}\label{re.hyp}
  Note that the factorization property implies {\it hyperlinearity} (or {\em Connes' embedding property}) 
  for the respective discrete quantum group in the sense of
  \cite[$\S$3.2]{bcv}: the weak-$*$ closure of $\cA$ in the GNS
  representation of the Haar state is embeddable into an ultrapower of
  the hyperfinite $II_1$ factor.  This follows, for instance, from \cite[Proposition 3.2]{kirch} (see
  also \cite[Exercise 6.2.4]{bo} and the discussion on
  \cite[p. 198]{thom-hyp}).
\end{remark}

\begin{remark}\label{re.freeqg}
	The third named author (A.C.) showed in \cite{chi-rf} that the free discrete quantum groups 
	$\widehat{U_n^+}$ and $\widehat{O_n^+}$ are RFD when $n\neq 3$. 
	Along with \Cref{th.RFD_FP} above, 
	this implies the main result in \cite{BW} of the first (A.B.) and last (S.W.) named authors 
	stating that $\widehat{U_n^+}$ and $\widehat{O_n^+}$ have factorization property under the same condition. 
	Therefore, as remarked in \Cref{re.hyp} above, these discrete quantum groups are hyperlinear (or have 
	Connes' embedding property), which is the main result of the second named 
	author and his collaborators in \cite{bcv}. 	
\end{remark}

\section{Properties (F) and (T) imply residual finiteness}\label{se.ft}

Recall the notion of residual finiteness for quantum groups introduced
in \Cref{def.rf} and property (T), as recalled above in \Cref{def.t}. The main result of this section is the following theorem.

\begin{theorem}\label{th.tf}
  A discrete quantum group with property (T) and the factorization property is residually finite.
\end{theorem}

Before going into the proof, let us fix our notation for a discrete
quantum group $\Gamma=\widehat{G}$ as above. We denote as usual by
$\cA=\cA(G)$ its CQG algebra, and by $\pi:\cA\to B(H)$ a
universal representation of the $C^*$-envelope $C^u(G)$ of $\cA$ on a Hilbert
space $H$. Our choice of $\pi$ is such that every UCP map $\psi:\cA\to M_n$
can be written as
\begin{equation*}
  \psi(\cdot)=T^*\pi(\cdot) T\text{ for an isometry } T:\bC^n\to H. 
\end{equation*}

For any two representations $\pi_1, \; \pi_2$ of a Hopf $\ast$-algebra $\cA$ on Hilbert spaces $H_1, \; H_2$, respectively, one can form the the tensor 
product representation $\pi_1 \otimes \pi_2: \cA \to B(H_1\otimes H_2)$, which is given (by abuse of notation) \[(\pi_1 \otimes \pi_2)(a):= (\pi_1 \otimes \pi_2)(\Delta a)  \qquad (a \in A),\] where $\Delta: \cA \to \cA \otimes \cA$ is the coproduct.

Also, for a CQG algebra $\cA$ associated to a unimodular discrete quantum group with antipode $S$ (which therefore extends to a $\ast$-anti-automorphism for the universal $C^*$-norm), 
and a $\ast$-representation $\pi$ of the $\ast$-algebra $\cA$ on a Hilbert space $H$, 
the dual representation $\pi^*$ of $\pi$ on the conjugate Hilbert space $\bar{H}$ is defined by 
$$
\pi^*(a) := J \pi(S(a))^* J^{-1}, \; \; a \in \cA
$$
where $J: H \rightarrow \bar{H}$ is the conjugation operator, and the dual space $H^*$ is identified 
with $\bar{H}$ as usual via Riesz representation. 
(Cf. the notion of contragradient representation in \cite{Wor87}.)

We will also regard the space $H\otimes H^*\cong \mathcal{HS}(H)$ of
Hilbert-Schmidt operators on $H$ as the underlying Hilbert space of
the tensor product representation $\pi\otimes \pi^*$ of the
Hopf $*$-algebra $\cA$. Under this identification, a vector $w\in H\otimes H^*$ is fixed under
$\pi\otimes \pi^*$ precisely when it is a
$\pi$-intertwiner.

We are now ready to address \Cref{th.tf}.

\begin{proof_of_tf}
  We begin by using property (F) to select a net $(\varphi_k)_k$ of
  UCP maps
\begin{equation*}
  \varphi_k:\cA\to M_{n_k} 
\end{equation*}
approximating the Haar state $\tau:\cA\to \bC$ as in part (2) of \Cref{thm-f}:
\begin{equation}\label{eq:net-mult}
  \|\varphi_k(a^*b)-\varphi_k(a)^*\varphi_k(b)\|_{2,n_k}\to 0
\end{equation}
where $\|x\|_{2,n}=\tr_n(x^*x)^{\frac 12}$ for $x\in M_n$ and
\begin{equation}\label{eq:net-tr}
  \tr_{n_k}\circ \varphi_k\to \tau
\end{equation}
pointwise. 

As described above, our choice of $\pi:\cA\to B(H)$ gives rise to isometries $T_k$ with
\begin{equation*}
  T_k:\bC^{n_k}\to H,\ \varphi_k(\cdot) = T_k^* \pi (\cdot) T_k. 
\end{equation*}


The almost-multiplicativity \Cref{eq:net-mult} of the net $(\varphi_k)$ can then be recast as follows.

Let $H_x$ be the $n$-dimensional carrier space of a unitary $u=u^x\in M_n(\cA)\cong B(H_x)\otimes \cA$ as explained in \Cref{se.prel} and consider the inflated maps
\begin{equation*}
 \psi_k:=\mathrm{id}\otimes \varphi_k:B(H_x)\otimes \cA\to M_{n_k\times n} 
\end{equation*}
where $ M_{n_k\times n}:= B(H_x) \otimes M_{n_k} \cong M_n(M_{n_k})$.
These are again UCPs and satisfy their own version of \Cref{eq:net-mult}, by simply inflating the latter:

\begin{equation}\label{eq:net-mult-bis}
    \|\psi_k(a^*b)-\psi_k(a)^*\psi_k(b)\|_{2,n_k\times n}\to 0
  \end{equation}
for $a,b\in B(H_x)\otimes \cA$.   

In particular, applying this to $a=u^*=b$, we obtain
\begin{equation}\label{eq:tr0}
  \tr_{n_k\times n}\left(1-\psi_k(u)\psi_k(u)^*\right)\to 0.
\end{equation}
Let
\begin{itemize}
\item $U$ denote the image of $u\in B(H_x)\otimes \cA$ through
\begin{equation*}
  \mathrm{id}\otimes\pi: B(H_x)\otimes \cA \to B(H_x)\otimes B(H). 
\end{equation*}
\item
  \begin{equation*}
    w_k=\frac{T_kT^*_k}{\sqrt{n_k}} 
  \end{equation*}
  so that $1\otimes w_k$ is a rescaled finite-rank projection in $B(H_x)\otimes B(H)$ of Hilbert-Schmidt norm $\sqrt n$.
  \item $q_k=1\otimes T_kT_k^*$. 
\end{itemize}
Reprising the computation in \cite[proof of Proposition 2.3]{kirch} with $\mathrm{Tr}$ standing for un-normalized traces, we obtain
\begin{align*}
  \tr_{n_k\times n}\left(1-\psi_k(u)\psi_k(u)^*\right) &=\frac 1{nn_k}\mathrm{Tr}\left(1-\psi_k(u)\psi_k(u)^*\right)\\
                                                       &=\frac 1{nn_k}\mathrm{Tr}\left(q_k - q_k U q_k U^*\right)\\
                                                       &=\frac 1{nn_k}\left(\|q_k\|^2_{HS}-\braket{q_k,U q_k U^*}_{HS}\right)\\
                                                       &=\frac 12\left\|\frac 1{\sqrt{nn_k}}(q_k-U q_k U^*)\right\|^2_{HS}.\\
\end{align*}
Since $q_k=\sqrt{n}(1\otimes w_k)$, this gives us
\begin{equation}\label{eq:tr-kirch}
  \tr_{n_k\times n}\left(1-\psi_k(u)\psi_k(u)^*\right) = \frac 1{2n}\left\|1\otimes w_k-U(1\otimes w_k)U^*\right\|^2_{HS}.
\end{equation}

\Cref{eq:tr0,eq:tr-kirch} now imply that the right hand side of \Cref{eq:tr-kirch} converges to zero, or equivalently
\begin{equation}\label{eq:1wu}
1\otimes w_k-U(1\otimes w_k)U^*\to 0\quad\text{in}\quad H_x\otimes H_x^*\otimes H\otimes H^*.
\end{equation}
For each $\eta\in H_x$ the map $H_x\otimes H_x^*\to H_x$ defined by
\begin{equation*}
  H_x\otimes H^*_x\cong B(H_x)\ni T\mapsto T\eta\in H_x
\end{equation*}
is continuous. Applying it to the $H_x\otimes H^*_x$ tensorand in \Cref{eq:1wu} we obtain
\begin{equation}\label{eq:etaw}
  \|\eta\otimes w_k-U^x(\mathrm{id}\otimes w_k)(U^x)^*(\eta\otimes\mathrm{id})\|\to 0,
\end{equation}
where the norm is taken in the Hilbert space $H_x\otimes H\otimes H^*$ and we have written $U^x$ for $U$. 

Now note that the second term inside the norm in \Cref{eq:etaw} is simply the action of $u^x$ on $\eta\otimes w_k$ through $\pi\otimes\pi^*$.
In conclusion, \Cref{eq:etaw} translates to $(w_k)$ providing almost containment of invariant vectors for $\pi\otimes\pi^*$ in the sense of \Cref{def.t}.

Property (T) now ensures that for any $\varepsilon>0$ we can find an index $k$ and a Hilbert-Schmidt operator $w\in H\otimes H^*$, fixed by $\cA$ via $\pi\otimes\pi^*$, such that
\begin{equation*}
  \|w-w_k\|<\varepsilon\text{ in }H\otimes H^*.
\end{equation*}
Note that we are implicitly using the `continuity constants' version of property (T) (\Cref{re.cont-const}).

As noted above, being $(\pi\otimes\pi^*)$-fixed implies that $w$, regarded as an operator on $H$, is a $\pi$-intertwiner. This means that we can further approximate it by finite-rank $\pi$-intertwiners arbitrarily well: Indeed, decomposing the positive, compact operator $w_k$ as $\int_{[0,\infty)}E_\lambda\ \mathrm{d}\lambda$ via its resolution of the identity provided by the spectral theorem, we can simply substitute for $w_k$ the finite-rank operator
\begin{equation*}
  w_k = \int_{\left[\frac 1m,\infty\right)} E_\lambda\ \mathrm{d}\lambda
\end{equation*}
for sufficiently large $m$. We abuse notation slightly and denote such finite-rank approximants by $w$ again.  Finally, the faithfulness of the Haar state $\tau$ on $\cA$ and \Cref{eq:net-tr} imply that the finite-dimensional representations
\begin{equation}\label{eq:fd-reps}
  P_w\pi(-) P_w:\cA\to B(w H) 
\end{equation}
with
\begin{equation*}
  P_w:=\text{range projection of }w
\end{equation*}
for finite-rank $w\in H\otimes H^*$ as above separate the elements of $\cA$. Indeed \Cref{eq:net-tr} says that if $\Tr$ denotes the standard (non-normalized) trace on $B(H)$, then
\begin{equation}\label{eq:wk-to-a}
  \braket{w_k,\pi(a)w_k}_{HS} = \Tr(w_k \pi(a) w_k)\to \tau(a)
\end{equation}
for arbitrary $a\in \cA$, where the HS (Hilbert-Schmidt) inner product is 
\begin{equation*}
  \braket{x,y}_{HS} = \Tr(x^*y). 
\end{equation*}

In general, for two vectors $\xi,\eta$ in a Hilbert space acted upon by the bounded operator $T$, we have
\begin{align*}
  |\braket{\xi,T\xi}-\braket{\eta,T\eta}|&=|\braket{\xi,T\xi}- \braket{\eta,T\xi}+\braket{\eta,T\xi}-\braket{\eta,T\eta}|\\ &= |\braket{\xi-\eta,T\xi}+\braket{\eta,T(\xi-\eta)}|. 
\end{align*}
The last term is dominated by
\begin{equation*}
  |\braket{\xi-\eta,T\xi}|+|\braket{\eta,T(\xi-\eta)}|\le \|T\|(\|\xi\|+\|\eta\|)\|\xi-\eta\|
\end{equation*}
(a similar inequality is noted in \cite[proof of Proposition 2.3]{kirch} for unitary $T$ and unit vectors $\xi$, $\eta$).


Together with the fact that  $\|w-w_k\|_{HS}<\varepsilon$ for small $\varepsilon$ and $\|w\|_{HS}=\|w_k\|_{HS}=1$ this returns
\begin{equation*}
  |\braket{w_,\pi(a)w}_{HS} - \braket{w_k,\pi(a)w_k}_{HS}|<2\varepsilon. 
\end{equation*}
It follows from \Cref{eq:wk-to-a} that $\tau(a)$ can be approximated arbitrarily well by
\begin{equation*}
  \braket{w,\pi(a)w}_{HS} = \Tr(w \pi(a) w)
\end{equation*}
 with finite-rank intertwiners $w$ as above. In particular, for every $a\ne 0$ there is some such $w$ for which the operator $P_w\pi(a^*a)P_w\in B(wH)$ does not vanish, hence the claim that the representations \Cref{eq:fd-reps} form a separating family.   

Since separability by finite-dimensional $*$-representations is precisely the residual finiteness requirement of \Cref{def.rf}, this concludes the proof. 
\end{proof_of_tf}

\section{Extensions and bicrossed products}\label{se.ap}

In this section we prove residual finiteness for certain discrete quantum groups constructed in \cite{bicr}. For background on bicrossed products we refer to \cite[Section 3]{bicr}, and offer only a brief recollection here.

Let $G$ and $\Gamma$ be a compact and discrete group respectively, forming a {\it matched pair} in the sense that they are realized as trivially-intersecting closed subgroups of a locally compact group $H$, with the property that the product $\Gamma G$ has full Haar measure in $H$. This amounts to giving a left action $\alpha$ of $\Gamma$ on $G$ and a right action $\beta$ of $G$ on $\Gamma$ satisfying certain compatibility conditions (e.g. \cite[Proposition 3.3]{bicr}).

To each quadruple $(\Gamma,G,\alpha,\beta)$ as above one can attach a compact quantum group $\bG$, as explained in \cite{vv} or \cite[$\S$3.2]{bicr}; we denote it by $\bG(\Gamma,G,\alpha,\beta)$ when we wish to be explicit about the matched pair structure, and reserve the present notation of $\Gamma$, $G$, $\alpha$, $\beta$ and $\bG$ throughout the current section.

The quantum groups $\bG$ (or rather their discrete duals) will provide, under certain circumstances, examples possessing the properties we have been concerned with throughout this paper. Note that according to the construction of bicrossed products (e.g. as in \cite[$\S$3.2]{bicr}), the CQG algebra $\cA=\cA(\bG)$ is simply the crossed product $\cA(G)\rtimes \Gamma$ with respect to the action of $\Gamma$ induced by $\alpha$; the coalgebra structure does not feature in any crucial capacity here.

We will need the following piece of terminology.

\begin{definition}\label{def.fin-orb}
  An action $\alpha$ of a discrete group $\Gamma$ on a compact Hausdorff topological space $X$ {\it has enough finite orbits} if the points of $X$ with finite orbit under the action form a dense subset of $X$.
\end{definition}

We now have

\begin{theorem}\label{th.bicr-rf}
  Let $(\Gamma,G,\alpha,\beta)$ be a matched pair. The following conditions are equivalent:
  \begin{enumerate}
    \renewcommand{\labelenumi}{(\arabic{enumi})}
  \item $\cA(\bG)$ is RFD;
  \item the action $\alpha$ has enough finite orbits and the group algebra $\bC\Gamma$ is RFD.
  \end{enumerate}
\end{theorem}

Let us first record the following immediate consequence.

\begin{corollary}\label{cor.bicr-rf}
  Under the hypotheses of \Cref{th.bicr-rf}, if $\Gamma$ is finitely
  generated and $G$ is a Lie group, then $\widehat{\bG}$ is residually
  finite in the sense of \Cref{def.rf}.
\end{corollary}
\begin{proof}
  Indeed, according to \Cref{def.rf} (and given \Cref{th.bicr-rf}) all that is missing is the finite generation of the algebra $\cA(\bG)$, which follows under the circumstances:

  Our hypothesis ensures that both $\bC \Gamma$ and the algebra $\cA(G)$ of representative functions on $G$ are finitely generated, and the conclusion follows from $\cA(\bG)\cong \cA(G)\rtimes \Gamma$. 
\end{proof}

\begin{proof_of_brf} We prove the two implications separately.

  \vspace{.5cm}

  {\bf (1) $\Rightarrow$ (2)} The RFD-ness of $\bC\Gamma$ follows from that of $\cA(\bG)$, since the former is a $*$-subalgebra of the latter.

  As for the finite-orbits condition, we can argue as follows. For any open subset $U\subset G$ a continuous function on $G$ with support in $U$ is 
  not annihilated in some finite dimensional representation of the full crossed product algebra $\pi:C(G)\rtimes \Gamma\to M_m(\bC)$.  It follows that some one-dimensional representation $\chi$ of $C(G)$ supported in $U$ is contained in $\pi$. Since the entire $\alpha$-orbit of $\chi$ is then contained in $\pi$ by the $\Gamma$-equivariance of the representation, that orbit must be finite.

  \vspace{.5cm}

  {\bf (2) $\Rightarrow$ (1)} Consider an $\alpha$-invariant finite subset $F\subset G$, and let $N\trianglelefteq \Gamma$ be the (finite-index) kernel of the morphism $\Gamma\to S_F$ into the symmetric group on the finite set $F$.

  Our assumption of residual finiteness for $\Gamma$ implies that its elements are separated by finite quotients
  \begin{equation*}
    \Gamma\to \Gamma_i.    
  \end{equation*}
Considering the resulting product morphisms $\Gamma\to \Gamma_i\times S_F$ instead, we may as well assume that the kernels of $\Gamma\to \Gamma_i$ are contained in $N$. But then the elements of the crossed product $C(F)\rtimes \Gamma$ are separated by homomorphisms of the form
  \begin{equation*}
    C(F)\rtimes \Gamma\to C(F)\rtimes \Gamma_i. 
  \end{equation*}
  onto finite-dimensional $C^*$-algebras. Indeed, since the underlying vector space of $C(F)\rtimes \Gamma$ is simply the tensor product $C(F)\otimes \bC\Gamma$, a non-zero element $x\in C(F)\rtimes \Gamma$ can be written uniquely as a non-empty sum
  \begin{equation*}
    \sum_{j}x_j\otimes g_j,\ x_j\ne 0\in C(F)
  \end{equation*}
  for distinct $g_j\in \Gamma$. Now simply choose $\Gamma\to \Gamma_i$ so that the images of $g_j$ are distinct.  
  
  Finally, the condition of having enough finite orbits ensures that homomorphisms of the form
  \begin{equation*}
    \cA\cong \cA(G)\rtimes \Gamma\to C(F)\rtimes \Gamma
  \end{equation*}
  separate the elements of $\cA$. This concludes the proof.
\end{proof_of_brf}

Examples of actions of residually finite discrete groups on compact (Lie) groups that do not meet the requirements of \Cref{th.bicr-rf} are easily constructed:

\begin{example}\label{ex.fin-orb}
  Let $G$ be a unitary group $U_n$, $n\ge 2$ and $\alpha$ the action of $\Gamma=\bZ$ via conjugation by an element $x\in U_n$ that is sufficiently generic, in the sense that its eigenvalues $\lambda_i$ satisfy $\lambda_i^m\ne \lambda_j^m$ for all $i\ne j$ and $m\in \bZ\setminus\{0\}$.

  The only elements of $U_n$ with finite orbit under $\alpha$ are those that commute with some power $x^m$, $m\in \bZ\setminus\{0\}$, and hence preserve all eigenspaces of $x$. Certainly, this is not a dense subset of $U_n$.
\end{example}

According to \Cref{th.RFD_FP}, we now also have

\begin{corollary}\label{cor.bicr-fp}
  If $G$ is finite and $\Gamma$ is residually finite, then the discrete quantum group $\widehat{\bG}=\widehat{\bG}(\Gamma,G,\alpha,\beta)$ has property (F).  \qedhere
\end{corollary}

On the other hand, \cite{bicr} also provides us with examples fitting into the setup of \Cref{se.ft}:

\begin{corollary}\label{cor.all}
  Suppose $G$ is finite and $\Gamma$ has properties (T) and (F). Then, $\widehat{\bG}$ is residually finite and has properties (T) and (F).
\end{corollary}
\begin{proof}
  \cite[Theorem 4.3]{bicr} ensures that $\widehat{\bG}$ has property (T). On the other hand, \cite[Theorem 1.1]{kirch} shows that $\Gamma$ is residually finite, and hence \Cref{th.bicr-rf} above is applicable to prove that $\widehat{\bG}$ is RF. It then also has property (F) by \Cref{th.tf}.
\end{proof}

\Cref{th.bicr-rf} fits into the general framework of proving that certain properties for discrete quantum groups (in this case the RFD property) are preserved under taking {\it extensions}:
\begin{definition}\label{def.ext}
  Let $N$ and $K$ be discrete quantum groups with underlying group algebras $\cB=\cA(\widehat{N})$ and $\cC=\cA(\widehat{K})$. An {\it extension of $K$ by $N$} is a discrete quantum group $\Gamma$ with underlying group algebra $\cA=\cA(\widehat{\Gamma})$ fitting into an exact sequence
  \begin{equation}\label{eq:seq}
    \bC\to \cB\to \cA\to \cC\to \bC
  \end{equation}
  of Hopf ($*$-)algebras in the sense of \cite[Definition 1.2.0, Proposition 1.2.3]{ad}.
\end{definition}
See also \cite[p. 523]{Wan13} for a discussion of exact sequences in the dual context of {\it compact} quantum groups, which amounts to the same exactness condition imposed in \Cref{def.ext}.

\begin{remark}\label{re.freeqg}
  \Cref{def.ext} specializes to the usual notions of exactness and extension for ordinary (i.e. non-quantum) discrete groups, and \Cref{th.bicr-rf} provides sufficient conditions for a special class of extension (arising as a bicrossed product) of RFD discrete quantum groups to retain the RFD property. \Cref{ex.fin-orb}, however, shows that the RFD property is not inherited by crossed products from their factors. This contrasts with the situation for discrete groups, where split extensions (i.e. those expressible as crossed products) of residually finite groups by finitely-generated residually finite groups are again residually finite by \cite{mal-ext}.
\end{remark}

As far as property (F) is concerned, we have the following positive result.

\begin{proposition}\label{pr.fext}
  A semidirect product of a discrete group with property (F) by a residually finite and finitely generated discrete group again has property (F).
\end{proposition}
\begin{proof}
  According to \cite[Theorem 1]{arz-semidir} it suffices to argue that
  \begin{enumerate}
    \renewcommand{\labelenumi}{(\alph{enumi})}
  \item fully residually property (F) groups retain the property, and
  \item semidirect extensions of property (F) groups by finite kernels have (F),
  \end{enumerate}
  where we say that a group $\Gamma$ fully residually has a given property $\cP$ provided for every finite subset $F\subset\Gamma$ some morphism $\Gamma\to K$ to a group with property $\cP$ is one-to-one on $F$.

  \vspace{.5cm}

  {\bf (a)} To argue the first item, note first that property (F) is preserved under taking products. Indeed, if $\Gamma=\prod_i \Gamma_i$, then in the language of \cite{bcv} the compact quantum group $\widehat{\Gamma}$ is topologically generated by $\widehat{\Gamma}_i$, and the conclusion follows from a simple adaptation of \cite[Theorem 3.3]{BW} to more than two compact quantum subgroups generating an ambient compact quantum group.

  Next, it follows from the characterization of property (F) in terms of a net of almost-representations $\varphi_k$ (see the proof of \Cref{th.tf} above and \cite[Theorem 6.2.7]{bo}) that property (F) is preserved by passing to subgroups.

  Finally, being fully residually (F) entails embeddability into a product of property-(F) groups, hence the conclusion.
  
  \vspace{.5cm}

  {\bf (b)} Consider a semidirect product $\Gamma=N\rtimes K$ with $N$ finite and $K$ having property (F). Consider the subgroup $K'\subset K$ consisting of elements that fix $N$ pointwise. The product
  \begin{equation*}
    NK'\subset \Gamma
  \end{equation*}
  is direct and hence (F), and moreover $K'\subseteq K$ has finite index. To verify the latter assertion, observe that $K'$ is the kernel of morphism $K\to S_N$ into the symmetric group on $N$ induced by the permutation of $N$ by $K$-conjugation. 

  In conclusion, it suffices to argue that property (F) lifts from finite-index subgroups
  \begin{equation*}
    \Omega\subset \Gamma 
  \end{equation*}
  (i.e. if $[\Gamma:\Omega]<\infty$ and $\Omega$ has (F) then so does $\Gamma$). This follows for instance from the characterization of property-(F) groups by the requirement that the linear functional
  \begin{equation*}
    \mu:\bC\Gamma\otimes \bC\Gamma\to \bC
  \end{equation*}
  defined by $(\gamma,\eta)\mapsto \delta_{\gamma,\eta}$ be continuous with respect to the minimal tensor product norm (e.g. \cite[Theorem 6.2.7 (3)]{bo}) on $C^*(\Gamma) \otimes C^*(\Gamma)$. Given that this condition holds for the finite-index subgroup $\Omega$ of $\Gamma$, it holds for $\Gamma$:

  When equipped with the minimal tensor product norm from above, the topological vector space $\bC\Gamma\otimes \bC\Gamma$ is a direct sum of finitely many subspaces isomorphic to $\bC\Omega\otimes \bC\Omega$ (translates by $(\gamma,\eta)$ with $\gamma$ and $\eta$ ranging over a finite system of representatives for the cosets of $\Omega$ in $\Gamma$). Since the restriction of $\mu$ to all of these subspaces is continuous by assumption, the conclusion follows.
\end{proof}

\begin{remark}\label{re.fext}
  Note that the bicrossed product discrete quantum group 
  $\widehat{\bG}(\Gamma,G,\alpha,\beta)$  in
  \Cref{ex.fin-orb} has property (F). The reason is that the
  underlying $C^*$-algebra $C(U_n)\rtimes\bZ$ is nuclear (because
  $\bZ$ is amenable; see e.g. \cite[Theorem 4.2.6]{bo}) and hence
  there is no distinction between the maximal and minimal tensor
  products appearing in the original definition of the factorization
  property. Therefore, by \Cref{th.tf,th.bicr-rf}, 
  the quantum group  $\widehat{\bG}(\Gamma,G,\alpha,\beta)$
  in \Cref{ex.fin-orb}  does not have property (T). 
\end{remark}

We end the present section with a discussion of the preservation of
property (T) under extensions of discrete quantum groups in the sense
of \Cref{def.ext}. This is a natural question to pose, given the
classical version (e.g. \cite[Lemma 7.4.1]{zim} or \cite[Proposition
1.7.6]{t}).  The result we prove, analogous to its classical version,
involves the following notion (cf. \cite[Definition 1.4.3]{t}).

\begin{definition}\label{def.relt}
  Let $N\le \Gamma$ be an inclusion of discrete quantum groups in the
  sense that we have an embedding
  $\cB=\cA(\widehat{N})\to \cA(\widehat{\Gamma})=\cA$ of CQG algebras.   The pair $(\Gamma,N)$ {\it has property (T)} if every
  $\cA$-representation that almost contains invariant vectors admits a
  non-zero vector invariant under $\cB$.
\end{definition}

\begin{remark}
  \Cref{def.relt} agrees with \cite[Definition 4.1]{bicr} once one
  accounts for the fact that the latter is formulated in terms of the
  compact Pontryagin duals to the discrete quantum groups of interest
  here. 
\end{remark}

\begin{proposition}\label{pr.extt}
  Consider an exact sequence of discrete quantum groups as in
  \Cref{def.ext}. Then, $\Gamma$ has property (T) if and only if $K$
  and the pair $(\Gamma,N)$ do.
\end{proposition}

Note that property (T) for $\Gamma$ is equivalent to property (T) for the pair $(\Gamma,\Gamma)$, and (T) for $N$ entails the property for the pair $(\Gamma,N)$.

\begin{proof}
  The direct implication is immediate, so we argue the converse. Suppose, in other words, that $K$ and $(\Gamma,N)$ both have property (T), and consider a representation $\pi:\cA\to B(H)$ with almost invariant vectors.

  Property (T) for the pair then implies that the Hilbert subspace $H_0\subseteq H$ consisting of $N$-invariant vectors is non-zero. $H_0$ is moreover $\Gamma$-invariant because $N\le \Gamma$ is normal in the sense of \cite[Theorem 2.7]{Wan13}. To see this, recall first that the normality implies in particular that the kernel of the surjection $\cA\to \cC$ is the (left and right) ideal (cf. \cite[Lemma 3.3]{Wan13})
  \begin{equation*}
   \cA\ker(\varepsilon|_{\cB}) = \ker(\varepsilon|_{\cB}) \cA.  
  \end{equation*}
  This equality then implies that we have
  \begin{equation*}
    \pi(\ker(\varepsilon|_{\cB}) \cA) H_0 = \pi(\cA\ker(\varepsilon|_{\cB})) H_0 = 0,
  \end{equation*}
  because $H_0$ being $N$-invariant means that $\cB$ acts on $H_0$ via $\varepsilon$. This means that $\ker(\varepsilon|_{\cB}) \pi(\cA)H_0 = 0$, and hence the space $\pi(\cA)H_0$ is contained in the space $H_0$ of $N$-fixed vectors in $H$; i.e. $H_0$ is $\Gamma$-invariant.  Denote the representation of $\cA$ (i.e. $\Gamma$) on $H_0$ by $\pi_0$.

  Next, we claim that the $\Gamma$-representation $\pi_0$ on $H_0$ almost contains invariant vectors in the sense of \Cref{def.t}. To verify this, choose an $(X,\varepsilon)$-invariant unit vector $v_n\in H$ for each $n=(X,\varepsilon)$ where $X$ is a finite set of irreducible $\cA$-comodules and $\varepsilon>0$. These vectors form a net as $X$ exhausts $\ir(\widehat{\Gamma})$ and $\varepsilon\to 0$.

  Now suppose there is a subnet $v_m$ whose projections $v_m^\perp: =Pv_m$ on $H_0^\perp$ have norms bounded below by some $C>0$. According to \Cref{def.t,re.1phi}, for every element $x\in X$ and unit vector $\eta\in H_x$ we have
  \begin{equation}\label{eq:evn}
    U^x(\eta\otimes v_n)-(\eta\otimes v_n)\to 0. 
  \end{equation}
  $H_0^\perp$ is invariant under $\cA$ via $\pi$, and hence the projection $P$ with range $H_0^\perp$ commutes with $\cA$. Because $u^x$ belongs to $B(H_x)\otimes \cA$, applying
  \begin{equation*}
    \mathrm{id}\otimes(\text{projection onto }H_0^\perp)
  \end{equation*}
  to \Cref{eq:evn} yields
  \begin{equation*}
    U^x(\eta\otimes v_n^\perp)-(\eta\otimes v_n^\perp)\to 0.
  \end{equation*}
  Restricting to the subnet $(v_m)_m$ and using $\|v_m^\perp\|\ge C$
  we obtain
  \begin{equation*}
    U^x\left(\eta\otimes \frac{v_m^\perp}{\|v_m^\perp\|}\right)-\left(\eta\otimes \frac{v_m^\perp}{\|v_m^\perp\|}\right)\to 0. 
  \end{equation*}
  In conclusion, the normalized projections $\frac{v_m^\perp}{\|v_m^\perp\|}$ attest to the existence of almost invariant vectors for the representation of $\Gamma$ on $H_0^\perp$. Property (T) for the pair $(\Gamma,N)$ then entails the existence of $N$-invariant vectors in $H_0^\perp$. This, however, contradicts the choice of $H_0$ as the space of {\it all} $N$-invariant vectors in $H$.

  The contradiction we have just obtained shows that the norms of the projections $v_n'$ of $v_n$ on $H_0$ converge to $1$ along the net. The same argument (projecting onto $H_0$ instead of $H_0^\perp$) then shows that these projections witness the fact that the $\Gamma$-representation $\pi_0$ almost contains invariant vectors.
  
  As observed before, the trivial action of $N$ means that $\cB$ acts on $H_0$ via the counit $\varepsilon$. On the other hand, the exact sequence \Cref{eq:seq} implies that the kernel of $\cA\to \cC$ is the ideal of $\cA$ generated by $\mathrm{ker}(\varepsilon|_{\cB})$ (cf. \cite[Lemma 3.3]{Wan13}) which is annihilated by $\pi_0$ as noted above, and hence the representation $\pi_0:\cA\to B(H_0)$ factors through $\cC=\cA(\widehat{K})$. The existence of almost invariant vectors and property (T) for $K$ then implies that $H_0$ contains non-zero $K$-fixed vectors; these would then also be fixed by $\Gamma$, finishing the proof that the latter has property (T).
\end{proof}



\begin{thebibliography}{10}

\bibitem{ad}
N.~Andruskiewitsch and J.~Devoto.
\newblock Extensions of {H}opf algebras.
\newblock {\em Algebra i Analiz}, 7(1):22--61, 1995.

\bibitem{arz-semidir}
G.~{Arzhantseva} and {\'S}.~R. {Gal}.
\newblock {On approximation properties of semidirect products of groups}.
\newblock {\em ArXiv e-prints}, December 2013.

\bibitem{BS93}
Saad Baaj and Georges Skandalis.
\newblock Unitaires multiplicatifs et dualit\'e pour les produits crois\'es de
  {$C^*$}-alg\`ebres.
\newblock {\em Ann. Sci. \'Ecole Norm. Sup. (4)}, 26(4):425--488, 1993.

\bibitem{t}
Bachir Bekka, Pierre de~la Harpe, and Alain Valette.
\newblock {\em Kazhdan's property ({T})}, volume~11 of {\em New Mathematical
  Monographs}.
\newblock Cambridge University Press, Cambridge, 2008.

\bibitem{BW}
Angshuman Bhattacharya and Shuzhou Wang.
\newblock Kirchberg's factorization property for discrete quantum groups.
\newblock {\em Bull. Lond. Math. Soc.}, 48(5):866--876, 2016.

\bibitem{bcv}
Michael Brannan, Beno{\^\i}t Collins, and Roland Vergnioux.
\newblock The {C}onnes embedding property for quantum group von {N}eumann
  algebras.
\newblock {\em Trans. Amer. Math. Soc.}, 369(6):3799--3819, 2017.

\bibitem{BrKe17}
Michael Brannan and David Kerr.
\newblock {Q}uantum groups, property ({T}), and weak mixing.
\newblock Commun. Math. Phys. (2017).
  https://doi.org/10.1007/s00220-017-3037-0, 2017.

\bibitem{bo}
Nathanial~P. Brown and Narutaka Ozawa.
\newblock {\em {$C^*$}-algebras and finite-dimensional approximations},
  volume~88 of {\em Graduate Studies in Mathematics}.
\newblock American Mathematical Society, Providence, RI, 2008.

\bibitem{chen-ng}
Xiao Chen and Chi-Keung Ng.
\newblock Property {$T$} for locally compact quantum groups.
\newblock {\em Internat. J. Math.}, 26(3):1550024, 13, 2015.

\bibitem{chi-rf}
Alexandru Chirvasitu.
\newblock Residually finite quantum group algebras.
\newblock {\em J. Funct. Anal.}, 268(11):3508--3533, 2015.

\bibitem{DFSW16}
Matthew Daws, Pierre Fima, Adam Skalski, and Stuart White.
\newblock The {H}aagerup property for locally compact quantum groups.
\newblock {\em J. Reine Angew. Math.}, 711:189--229, 2016.

\bibitem{dsv}
Matthew Daws, Adam Skalski, and Ami Viselter.
\newblock Around {P}roperty ({T}) for quantum groups.
\newblock {\em Comm. Math. Phys.}, 353(1):69--118, 2017.

\bibitem{hv-t}
Pierre de~la Harpe and Alain Valette.
\newblock La propri\'et\'e {$(T)$} de {K}azhdan pour les groupes localement
  compacts (avec un appendice de {M}arc {B}urger).
\newblock {\em Ast\'erisque}, (175):158, 1989.
\newblock With an appendix by M. Burger.

\bibitem{EfRu}
Edward~G. Effros and Zhong-Jin Ruan.
\newblock Discrete quantum groups. {I}. {T}he {H}aar measure.
\newblock {\em Internat. J. Math.}, 5(5):681--723, 1994.

\bibitem{bicr}
P.~{Fima}, K.~{Mukherjee}, and I.~{Patri}.
\newblock {On compact bicrossed products}.
\newblock {\em ArXiv e-prints}, March 2015.

\bibitem{fim}
Pierre Fima.
\newblock Kazhdan's property {$T$} for discrete quantum groups.
\newblock {\em Internat. J. Math.}, 21(1):47--65, 2010.

\bibitem{kazhdan}
D.~A. Ka\v{z}dan.
\newblock On the connection of the dual space of a group with the structure of
  its closed subgroups.
\newblock {\em Funkcional. Anal. i Prilo\v zen.}, 1:71--74, 1967.

\bibitem{kirch-orig}
Eberhard Kirchberg.
\newblock Positive maps and nuclear $c^*$-algebras.
\newblock In {\em Proc. Inter. Conference on Operator Algebras, Ideals and
  their Applications in Theoretical Physics}, pages 255--257, 1977.

\bibitem{kirch-ss}
Eberhard Kirchberg.
\newblock On nonsemisplit extensions, tensor products and exactness of group
  {$C^*$}-algebras.
\newblock {\em Invent. Math.}, 112(3):449--489, 1993.

\bibitem{kirch}
Eberhard Kirchberg.
\newblock Discrete groups with {K}azhdan's property {${\rm T}$} and
  factorization property are residually finite.
\newblock {\em Math. Ann.}, 299(3):551--563, 1994.

\bibitem{ky}
David Kyed.
\newblock A cohomological description of property ({T}) for quantum groups.
\newblock {\em J. Funct. Anal.}, 261(6):1469--1493, 2011.

\bibitem{ky-so}
David Kyed and Piotr~M. So\l~tan.
\newblock Property ({T}) and exotic quantum group norms.
\newblock {\em J. Noncommut. Geom.}, 6(4):773--800, 2012.

\bibitem{mal-lin}
A.~Malcev.
\newblock On isomorphic matrix representations of infinite groups.
\newblock {\em Rec. Math. [Mat. Sbornik] N.S.}, 8 (50):405--422, 1940.

\bibitem{mal-ext}
A.~Malcev.
\newblock On homomorphisms onto finite groups.
\newblock {\em Uchen. Zapiski Ivanovsk. ped. instituta}, 18 (5):49--60, 1958.

\bibitem{MontS}
Susan Montgomery.
\newblock {\em Hopf algebras and their actions on rings}, volume~82 of {\em
  CBMS Regional Conference Series in Mathematics}.
\newblock Published for the Conference Board of the Mathematical Sciences,
  Washington, DC; by the American Mathematical Society, Providence, RI, 1993.

\bibitem{oz}
Narutaka Ozawa.
\newblock About the {QWEP} conjecture.
\newblock {\em Internat. J. Math.}, 15(5):501--530, 2004.

\bibitem{pp}
Jesse Peterson and Sorin Popa.
\newblock On the notion of relative property ({T}) for inclusions of von
  {N}eumann algebras.
\newblock {\em J. Funct. Anal.}, 219(2):469--483, 2005.

\bibitem{PoWo90}
P.~Podle\'s and S.~L. Woronowicz.
\newblock Quantum deformation of {L}orentz group.
\newblock {\em Comm. Math. Phys.}, 130(2):381--431, 1990.

\bibitem{So05}
Piotr~M. So\l~tan.
\newblock Quantum {B}ohr compactification.
\newblock {\em Illinois J. Math.}, 49(4):1245--1270, 2005.

\bibitem{thom-hyp}
Andreas Thom.
\newblock Examples of hyperlinear groups without factorization property.
\newblock {\em Groups Geom. Dyn.}, 4(1):195--208, 2010.

\bibitem{Ti08}
Thomas Timmermann.
\newblock {\em An invitation to quantum groups and duality}.
\newblock EMS Textbooks in Mathematics. European Mathematical Society (EMS),
  Z\"urich, 2008.
\newblock From Hopf algebras to multiplicative unitaries and beyond.

\bibitem{vv}
Stefaan Vaes and Leonid Vainerman.
\newblock Extensions of locally compact quantum groups and the bicrossed
  product construction.
\newblock {\em Adv. Math.}, 175(1):1--101, 2003.

\bibitem{dae-discr}
A.~Van~Daele.
\newblock Discrete quantum groups.
\newblock {\em J. Algebra}, 180(2):431--444, 1996.

\bibitem{Wan95_1}
Shuzhou Wang.
\newblock Free products of compact quantum groups.
\newblock {\em Comm. Math. Phys.}, 167(3):671--692, 1995.

\bibitem{Wan13}
Shuzhou Wang.
\newblock Equivalent notions of normal quantum subgroups, compact quantum
  groups with properties {$F$} and {$FD$}, and other applications.
\newblock {\em J. Algebra}, 397:515--534, 2014.

\bibitem{Wor87}
S.~L. Woronowicz.
\newblock Compact matrix pseudogroups.
\newblock {\em Comm. Math. Phys.}, 111(4):613--665, 1987.

\bibitem{Wor98}
S.~L. Woronowicz.
\newblock Compact quantum groups.
\newblock In {\em Sym\'etries quantiques ({L}es {H}ouches, 1995)}, pages
  845--884. North-Holland, Amsterdam, 1998.

\bibitem{zim}
Robert~J. Zimmer.
\newblock {\em Ergodic theory and semisimple groups}, volume~81 of {\em
  Monographs in Mathematics}.
\newblock Birkh\"auser Verlag, Basel, 1984.

\end{thebibliography}
\bibliographystyle{plain}
\addcontentsline{toc}{section}{References}

\def\polhk#1{\setbox0=\hbox{#1}{\ooalign{\hidewidth
  \lower1.5ex\hbox{`}\hidewidth\crcr\unhbox0}}}

\end{document}